\documentclass[12pt]{amsart}
\usepackage{amscd}
\usepackage{amssymb,amsmath}
\usepackage{hyperref}
\usepackage{mathrsfs}
\usepackage{graphicx}
\usepackage{enumitem}
\usepackage{romannum}
\usepackage{mathtools}
\usepackage[utf8]{inputenc}
\usepackage{listings}
\usepackage{lipsum,eso-pic,xcolor}
\usepackage{lineno}
\usepackage{tikz}
\usetikzlibrary{arrows,decorations.pathmorphing,backgrounds,positioning,fit}
\usetikzlibrary{positioning}
\hypersetup{
	colorlinks=true,
	linkcolor=blue,}

\usetikzlibrary{trees}
\usetikzlibrary{arrows}

\def\reg{\operatorname{reg}}
\def\deg{\operatorname{deg}}

\def\max{\operatorname{max}}

\def\v{\operatorname{v}}

\newcommand{\m}{\mathfrak m}

\newcommand{\NN}{\mathbb{N}}

\newcommand{\PP}{\mathcal{P}}
\newcommand{\im}{\mathrm{im}}

\newtheorem{lemma}{Lemma}[section]
\newtheorem{corollary}[lemma]{Corollary}
\newtheorem{theorem}[lemma]{Theorem}
\newtheorem{proposition}[lemma]{Proposition}
\newtheorem{definition}[lemma]{Definition}
\newtheorem{remark}[lemma]{Remark}

\advance\headheight1.15pt

\begin{document}
	
	\pagenumbering{arabic}
	
	\title[The v-number and Waldschmidt Constant]{The slope of v-function  and Waldschmidt Constant} 
	
	\author[Manohar Kumar]{Manohar Kumar$^a$}
	\address{Department of Mathematics, Indian Institute of Technology
		Kharagpur, West Bengal, INDIA - 721302.}
	\email{manhar349@gmail.com}
	
	\author[Ramakrishna Nanduri]{Ramakrishna Nanduri$^b$}
	\address{Department of Mathematics, Indian Institute of Technology
		Kharagpur, West Bengal, INDIA - 721302.}
  \email{nanduri@maths.iitkgp.ac.in} 

  \author[Kamalesh Saha]{Kamalesh Saha$^c$}
	\address{Department of Mathematics, Chennai Mathematical Institute, Chennai, INDIA - 603103.}
	\email{ksaha@cmi.ac.in}

	\thanks{$^a$ Supported by PMRF fellowship, India}
	\thanks{$^b$ Corresponding author and supported by SERB grant No: CRG/2021/000465, India}
 \thanks{$^c$ Supported by NBHM Postdoctoral Fellowship, India}
	\thanks{AMS Classification 2020: 13F20, 13A02, 13F55, 05E40}

	%\keywords{Regularity of }

\begin{abstract}
In this paper, we study the asymptotic behaviour of the v-number of a Noetherian graded filtration $\mathcal{I}= \{I_{[k]}\}_{k\geq 0}$ of a Noetherian $\mathbb{N}$-graded domain $R$. Recently, it is shown that $\v(I_{[k]})$ is periodically linear in $k$ for $k \gg 0$. We show that all these linear functions have the same slope, i.e. $\displaystyle \lim_{k \rightarrow \infty}\frac{\v(I_{[k]})}{k}$ exists, which is equal to $\displaystyle \lim_{k \rightarrow \infty}\frac{\alpha(I_{[k]})}{k}$, where $\alpha(I)$ denotes the minimum degree of a non-zero element in $I$. In particular, for any Noetherian symbolic filtration $\mathcal{I}= \{I^{(k)}\}_{k\geq 0}$ of $R$, it follows that
$\displaystyle \lim_{k \rightarrow \infty}\frac{\v(I^{(k)})}{k}=\hat{\alpha}(I)$, the Waldschmidt constant of $I$. Next, for a non-equigenerated square-free monomial ideal $I$, we prove that $\v(I^{(k)}) \leq \reg(R/I^{(k)})$ for $k\gg 0$. Also, for an ideal $I$ having the symbolic strong persistence property, we give a linear upper bound on $\v(I^{(k)})$. As an application, we derive some criteria for a square-free monomial ideal $I$ to satisfy $\v(I^{(k)})\leq \reg(R/I^{(k)})$ for all $k\geq 1$, and provide several examples in support. In addition, for any simple graph $G$, we establish that  $\v(J(G)^{(k)}) \leq \reg(R/J(G)^{(k)})$ for all $k \geq 1$, and $\v(J(G)^{(k)}) = \reg(R/J(G)^{(k)})=\alpha(J(G)^{(k)})-1$ for all $k\geq 1$ if and only if $G$ is a Cohen-Macaulay very-well covered graph, where $J(G)$ is the cover ideal of $G$. 
\end{abstract}

\maketitle 

\section{Introduction}
Let $\mathbb{N}$ denote the set of non-negative integers. Let $R=\bigoplus_{d\in\mathbb{N}}R_d$ be an $\mathbb{N}$-graded Noetherian ring and $I\subseteq R$ be a proper graded ideal. Then any $\mathfrak{p}\in\mathrm{Ass}(I)$ is of the form $I:f$ for some homogeneous $f\in R$, where $\mathrm{Ass}(I)$ denotes the set of associated primes of $I$. This fact defines the \textit{$\v$-number} of $I$, denoted by $\v(I)$, as follows:
\begin{align*}
    \v(I) := \min\{ d \geq 0 \mid \text{ there exists }\, f \in R_d \text{ and } \mathfrak{p} \in \text{Ass}(I) \text{ satisfying } I : f = \mathfrak{p} \}.
\end{align*}
For each $ \mathfrak{p} \in \text{Ass}(I) $, we can locally define the $\v$-number, called the \textit{local} $\v$-number of $I$ at $\mathfrak{p}$ and denoted by $\v_{\mathfrak{p}}(I)$, as follows:
\begin{align*}
 \v_{\mathfrak{p}}(I) := \min\{ d \geq 0 \mid \text{ there exists } \, f \in R_d \text{ satisfying } I : f = \mathfrak{p} \}.    
\end{align*}

The invariant $\v$-number was introduced in 2020 \cite{cstpv20} to investigate the asymptotic behaviour of the minimum distance function of projective Reed-Muller-type codes. The emergence of this invariant has sparked a new wave of exploration in the field of commutative algebra, resulting in several noteworthy applications. In addition, when viewed from a geometric standpoint, the degree of a truncator for a finite set of projective points, as discussed in \cite{gkr93}, is connected to the local v-number. In recent times, there has been a substantial amount of research on the $\v$-number (see \cite{ass23,as24,bm23,bms24,c23, co23,fs23, fs24, grv21, jv21, ss22, s23,ks23}). \par 

This paper focuses on studying the asymptotic behaviour of the $\v$-number for Noetherian graded filtrations of $R$. In this direction, Ficarra and Sgroi in \cite{fs23} and Conca in \cite{co23} independently proved that for a graded ideal $I$ in a polynomial ring over a field $\mathbb{K}$, the $\v$-function $\v(I^k)$ is an asymptotic linear function, i.e., $\v(I^k)$ is linear in $k$ for $k\gg 0$. In fact, Conca showed the result for a graded ideal $I$ in an $\mathbb{N}$-graded Noetherian domain. The $\v$-number of the $I$-adic filtration of modules is studied in \cite{fg24} and showed that its $\v$-number is asymptotically linear. Also, in \cite{bms24} and \cite{f23}, the authors identified many classes of graded ideals, where $\v(I^k)$ is a linear function from the very beginning. Recently, a generalization of this fact has been established for any Noetherian graded filtration $\mathcal{I}=\{I_{[k]}\}$ of an $\mathbb{N}$-graded Noetherian domain $R$ (see \cite{fs24}, \cite{as24}). In particular, it has been proved that $\v(I_{[k]})$ is eventually a quasi-linear function in $k$, i.e., $\v(I_{[k]})$ is periodically linear function in $k$ for $k\gg 0$. In this paper, our target is to show that $\displaystyle\lim_{k\rightarrow\infty} \frac{\v(I_{[k]})}{k}$ exists, i.e., the leading coefficients of all those linear functions, which represent $\v(I_{[k]})$ for large $k$, are same. Moreover, we explore the behaviour of the $\v$-function in the context of symbolic filtration of monomial ideals, establishing a relation between the $\v$-numbers of symbolic powers of certain monomial ideals and their corresponding (Castelnuovo-Mumford) regularity.

%We study the $\v$-number of Noetherian symbolic filtration $\{I^{(k)}\}_{k\geq 0}$ of a graded ideal $I$ in a graded domain $R$. 

%In \cite{jv21}, Jaramillo and Villarreal investigated the $\v$-number of edge ideals. For various classes of graphs, it has been demonstrated that $\v(I(G)) \leq \reg(R/I(G))$ (cf. \cite{jv21, grv21, ss22}). Later, Civan \cite{c23} established that $\v(I(G))$ can exceed $\reg(R/I(G))$ arbitrarily. However, in \cite{s23}, it was shown that for the cover ideals of graphs (the Alexander dual of edge ideals), the $\v$-number provides a precise lower bound for the regularity of the corresponding quotient ring. Furthermore, in \cite{ks23}, the authors examined scenarios where the $\v$-number may be equal to or greater than the regularity. Even in \cite{fs23}, authors showed that $\v$-number of ordinary powers of homogeneous ideals over the polynomial ring is asymptotic linear functions, and in \cite{co23}, Conca proves the same result over the domain. Thus, there is considerable interest in understanding the behaviour of the v-number in relation to other well-known invariants.

 %Recently, in \cite{fs24}, \cite{as24}, authors discuss that $\v$-number of Noetherian graded filtration of homogeneous ideals are quasi-linear. Also, they discuss partly existence of $\displaystyle \lim_{k \rightarrow \infty} \frac{\v_{\mathfrak{p}}(I_{[k]})}{k}$. In \cite{as24}, authors also discuss $\v$-number and regularity of symbolic powers of cover ideals of cyclic graphs, complete graphs, and complete bipartite graphs.  \\
 
Symbolic powers of homogeneous ideals and their algebraic properties are one of the central research topics in Commutative Algebra. During the last decade, significant attention has been devoted to the ``ideal containment problem": given a nontrivial homogeneous ideal $I$ of a Noetherian ring, the problem is to determine all positive integer pairs $(k, r)$ such that $I^{(k)} \subseteq I^r$, where $I^{(k)}$ denotes the $k$-th symbolic power of the ideal $I$ (see \cite{ddghn18}). To gain finer insights into these containment relationships, Bocci and Harbourne \cite{bh10} introduced $\rho(I) = \sup\{\frac{k}{r} \mid I^{(k)} \nsubseteq I^r \}$, known as the resurgence of $I$. In general, computing $\rho(I)$ is quite difficult. To predict $\rho(I)$, people investigate the Waldschmidt constant of $I$. Given any nonzero homogeneous ideal $I$ of $R$, the {\em Waldschmidt constant} of $I$ is $\hat{\alpha}(I): = \displaystyle\lim_{k \to \infty} \frac{\alpha(I^{(k)})}{k}$, where $\alpha(I) = \min\{d \mid I_d \neq 0\}$. Waldschmidt constants manifest in diverse forms across multiple mathematical disciplines. Initially, they garnered attention in complex analysis concerning the bounding of growth rates for holomorphic functions; refer to \cite{w77}. Bocci and Harbourne’s result $\frac{\alpha(I)}{\hat{\alpha}(I)} \leq \rho(I)$ (see \cite[Theorem 1.2]{bh10}) has drawn interest towards the computation of $\hat{\alpha}(I)$ in the realm of commutative algebra.
\par 

In this article, we pay special attention to the symbolic filtration of graded ideal $I$ in $R$ compared to other filtrations. When the symbolic Rees algebra $\mathcal{R}_{s}(I)$ is Noetherian, our result shows that $\displaystyle \lim_{k \rightarrow \infty}\frac{\v(I^{(k)})}{k}=\hat{\alpha}(I)$. Specifically, we analyze the $\v$-number of symbolic powers of square-free monomial ideals. On the other hand, the regularity of symbolic powers has been a topic of interest for numerous researchers (see \cite{f18}, \cite{f21}, \cite{n22}). By comparing the $\v$-number of symbolic powers of certain monomial ideals with their corresponding regularity, we partially answer the question raised in \cite[Question 5.1]{fs24}. Let's denote the cover ideal of a simple graph $G$ as $J(G)$, with the corresponding ring $R$. For all $k\geq 1$, we show that $\v(J(G)^{(k)})$ serves as a lower bound for $\reg(R/J(G)^{(k)})$. Additionally, we provide the necessary and sufficient conditions for equality. The paper is structured as follows.

%characterize the simple graphs $G$ which has $\v$-number and regularity of all symbolic powers of $J(G)$, are equal. The paper is organized in the following fashion.
\par 

In Section \ref{secpreli}, we recall some definitions, notions, and results associated with our work. In Section \ref{limit-existence}, we first prove in Lemma \ref{limit-exists} that if $\mathcal{I}=\{I_{[k]}\}$ is a Noetherian graded filtration of a Noetherian $\mathbb{N}$-graded domain $R$, then $\displaystyle \lim_{k \rightarrow \infty}\frac{\alpha(I_{[k]})}{k}$ exists. Using Lemma \ref{limit-exists} and other techniques, we establish our main result:\medskip

\noindent \textbf{Theorem \ref{limit-vnumber}.} \textit{Let $\mathcal{I}=\{I_{[k]}\}_{k \geq 0}$ be a Noetherian graded filtration of a Noetherian $\mathbb{N}$-graded domain $R$. Then, $\displaystyle \lim_{k \rightarrow \infty}\frac{\v(I_{[k]})}{k}$ exists and $\displaystyle\lim_{k \rightarrow \infty}\frac{\v(I_{[k]})}{k} =\lim_{k \rightarrow \infty}\frac{\alpha(I_{[k]})}{k}.$}
\medskip
 
 \noindent As a consequence of Theorem \ref{limit-vnumber}, we get $\displaystyle \lim_{k \rightarrow \infty}\frac{\v(I^{(k)})}{k}=\hat{\alpha}(I)$, when $\mathcal{R}_{s}(I)$ is Noetherian (Corollary \ref{waldscmidt-constant}). In section \ref{square-free-monomial-ideal}, we first give a sufficient condition in Proposition \ref{alpha-delta} for a monomial ideal $I$ to satisfy $\v(I^{(k)})\leq \reg (R/I^{(k)})$ for all $k \gg 0$. As an application of Proposition \ref{alpha-delta}, we show that if $I$ is a non-equigenerated square-free monomial ideal, then $\v(I^{(k)}) \leq \reg(R/I^{(k)})$ for all $k \gg 0$ (Corollary \ref{nonequi}). Next, for an ideal $I$ having symbolic strong persistence property, we establish an upper bound on $\v(I^{(k)})$ for all $k\geq 1$, which is a linear function in $k$ (see Theorem \ref{thm_vsym_ub}). An immediate consequence of this result gives a linear upper bound on the $\v$-number of symbolic powers of square-free monomial ideals, as stated in Corollary \ref{corsqfreeub}. In Proposition \ref{prop_poly}, we show that if $I$ is a square-free polymatroidal ideal, then $\v(I^{(k)})\leq \reg(R/I^{(k)})$ for all $k\geq 1$, and also discuss when the equality will hold. Also, we give some criteria in Proposition \ref{propvimreg} for a square-free monomial ideal $I$ to satisfy $\v(I^{(k)})\leq \reg(R/I^{(k)})$ and $\v(I^{k})\leq \reg(R/I^{k})$ for all $k\geq 1$, and in Remark \ref{bipartite-chordal}, we mention some classes satisfying those criteria. Using the criteria given in Section \ref{square-free-monomial-ideal}, we cannot argue the relation between $\v$-number and regularity for cover ideals of graphs. Thus, in Section \ref{cover-ideal}, using different techniques, we prove that $\v(J(G)^{(k)}) \leq \reg(R/J(G)^{(k)})$ for all $k\geq 1$ (Theorem \ref{v-numver-less-regularity}), where $G$ is any simple graph. Moreover, we show in Theorem \ref{v-number-equal-regularity} that for a simple graph $G$, $\v(J(G)^{(k)}) = \reg(R/J(G)^{(k)})=\alpha(J(G)^{(k)})-1$ for all $k\geq 1$ if and only if $G$ is Cohen-Macaulay very-well covered. At the end, we discuss some possible values of the slope of the $\v$-function for symbolic powers of cover ideals of graphs.

 %As an application, we get a linear upper bound of the $\v$-number of symbolic powers of square-free monomial ideals in Corollary \ref{corsqfreeub}
 %In the last, we compute explicitly $\displaystyle \lim_{k \rightarrow \infty}\frac{\v(J(G)^{(k)})}{k}$ of cover ideals  $J(G)$ of some classes of graphs $G$ (Corollary \ref{minimum-vertex-cover}).

\section{Preliminaries}\label{secpreli}
In this section, we recall some definitions and results which will be used throughout this paper.  
 \begin{definition} {\rm For a graded ideal $I$ in $R$, the\textit{ Castelnuovo-Mumford regularity} (or simply \textit{regularity}) of $R/I$, denoted by $\reg(R/I)$, is defined as 
 \begin{eqnarray*}
 \reg(R/I) &=&  \max \{j - i \mid \beta_{i,j}(R/I) \neq 0\} \\
         &=& \max\{j+i \mid H_{\m}^i(R/I)_j \neq 0\},   
 \end{eqnarray*}
  where $\beta_{i,j}(R/I)$ is the $(i,j)^{th}$ graded Betti number of $R/I$ and $H_{\m}^i(R/I)_j$ denotes the $j^{th}$ graded component of the $i^{t h}$ local cohomology module $H_{\m}^i(R/I)$.
  } 
\end{definition}

Let $R$ be a $\NN$-graded Noetherian domain. A \textit{graded filtration} of $R$ is a family $\mathcal{I} = \{I_{[k]}\}_{k\geq 0}$ of graded ideals of $R$ satisfying: (i) $I_{[0]} = R$; (ii) $I_{[k+1]} \subseteq I_{[k]}$ for all $k \geq 0$; (iii) $I_{[k]}I_{[r]} \subseteq I_{[k+r]}$ for all $k,r \geq 0$. We say that $\mathcal{I}$ is Noetherian if the Rees algebra $R(\mathcal{I})= \bigoplus_{k\geq 0} I_{[k]}t^k$ is a Noetherian ring. The ordinary powers of a graded ideal form a filtration of $R$. Also, the symbolic powers of a graded ideal give a filtration of $R$. 
In literature, the $k$-th symbolic power of an ideal $I$ of a Noetherian ring is defined in two ways:

 \begin{definition}[Symbolic power]\label{def-symbolic}
 {\rm Let $R$ be a Noetherian ring and $I\subset R$ be an ideal. Then, the  $k$-th symbolic power of $I$ is defined as
 \begin{enumerate}
     \item $ \displaystyle I^{(k)}=\bigcap_{P\in \mathrm{Ass}(R/I)}(I^k R_P \cap R).$
     \item $\displaystyle I^{(k)}=\bigcap_{P\in \mathrm{Min}(I)}(I^k R_P \cap R),$  where $\mathrm{Min}(I)$ is the set of all minimal prime of $I$.
 \end{enumerate}}
 \end{definition}
Throughout this paper, we use either definition unless otherwise specified. It is important to note that all the results derived in this article hold true for both definitions of symbolic powers. \par

%\begin{definition}[Differential power]
% Let $R$ be a Noetherian ring and $I\subset R$ be an ideal. For $f \in R$ and $x^a = x^{a_1}_1 \cdots x^{a_n}_n$, we denote $\frac{\partial f}{\partial x^a}$ as the partial derivative of $f$ with respect to $x^a$. 
% \begin{align*}
 % I^{\langle k \rangle} = \big{ f \in R \mid \frac{\partial f}{\partial x^a} \in I \text{ for all } x^a \text{ with } |a| \leq k - 1 \}.
%\end{align*}  
 %\end{definition}

%\begin{definition}[$\ast$-Differential power]
 %Let $R$ be a Noetherian ring and $I\subset R$ be an ideal. For $f \in R$ and $x^a = x^{a_1}_1 \cdots x^{a_n}_n$, we denote $\frac{\partial^{\ast}f}{\partial^{\ast}x^a}$ as the $\ast$-partial derivative of $f$ with respect to $x^a$, which is a derivative without coefficients. Then the $k$-th $\ast$-differential power of $I$ is defined as
 %\begin{align*}
  %I^{[k]} = \{ f \in S \mid \frac{\partial^{\ast}f}{\partial^{\ast}x^a} \in I \text{ for all } x^a \text{ with } |a| \leq k - 1 \},    
% \end{align*}
% \end{definition}

%In general, $\frac{\partial f}{\partial x^a} = c\frac{\partial^{\ast}f}{\partial^{\ast}x^a}$ for some constant $c$ and $I^{\langle k\rangle} \subseteq I^{[k]}$. When $I$ is a square-free monomial ideal, in \cite{m22}, authors prove that the symbolic powers of $I$ are equal to the $\ast$-differential powers of $I$.  
% \begin{lemma}\cite[Lemma 2.6]{m22} \label{differential-power}
 % Let $I$ be a square-free monomial ideal. Then $I^{(k)} = I^{[k]}$.   
 %\end{lemma}

\begin{definition}{\rm
Let $m=x_1^{a_1}\cdots x_n^{a_n}$ be a monomial in $R=\mathbb{K}[x_1,\ldots,x_n]$. Then the {\it polarization} of $m$ is defined to be the square-free monomial
\begin{equation*}
    \PP(m)=x_{11}x_{12} \cdots x_{1a_1} x_{21} \cdots x_{2a_2} \cdots x_{n a_n}
\end{equation*}
in the polynomial ring $\mathbb{K}[x_{i j} : 1 \leq j \leq a_i,1 \leq i \leq n ].$ If $I \subset R $ is a monomial ideal with the minimal generating set $\{m_1, \ldots, m_u\}$ and $\displaystyle m_i=\prod_{j=1}^{n}x_j^{a_{ij}}$, where each $a_{ij} \geq 0$ for $i=1, \ldots, u.$ Then the {\it polarization} of $I$, denoted by $ I^{\PP}$, is defined as 
\begin{equation*}
    I^{\PP}=(\PP(m_1), \ldots, \PP(m_u)),
\end{equation*}
which is a square-free monomial ideal in the polynomial ring $R^{\PP}=\mathbb{K}[x_{j1},x_{j2}, \ldots, x_{ja_j} \mid j=1, \ldots, n ]$, where $a_j=\max\{a_{i j} \mid i=1, \ldots,u \}$ for any $1 \leq j \leq n.$
}
\end{definition}

\begin{lemma}\cite[Corollary 3.5.]{ss22} \label{embedded-primes}
 For a monomial ideal $I$, we have $\v(I^{\PP}) = \v(I)$.  
\end{lemma}
\begin{definition}
{\rm A {\it hypergraph} $\mathcal{H}$ is a pair of two sets $(V(\mathcal{H}), E(\mathcal{H}))$, where $V(\mathcal{H})$ is called the vertex set of $\mathcal{H}$ and $E(\mathcal{H})$ is a collection of subsets of $V(\mathcal{H})$, called the edge set of $\mathcal{H}$. A \textit{simple} hypergraph is a hypergraph such that no two elements (called edges) of $E(\mathcal{H})$ contain each other. A simple graph is an example of a simple hypergraph, whose edges are of cardinality two.}
\end{definition}
Note that the term hypergraph in this manuscript always denotes a simple hypergraph.
Let $\mathcal{H}$ be a hypergraph on the vertex set $V(\mathcal{H}) = \{x_1, \ldots, x_n\}$. Then, for $A \subseteq V(\mathcal{H})$, we consider $X_A := \prod_{x_i \in A} x_i$ as a square-free monomial in the polynomial ring $R = \mathbb{K}[x_1, \ldots, x_n]$ over a field $\mathbb{K}$. The edge ideal of the hypergraph $\mathcal{H}$, denoted by $I(\mathcal{H})$, is the ideal in $R$ defined by
$$I(\mathcal{H}) = ( X_e \mid e \in E(\mathcal{H}) ),$$
where $X_e$ represents the square-free monomial corresponding to the edge $e$.
The {\em induced subhypergraph} $\mathcal{H'}$ of $\mathcal{H}$ on a subset $W$ of the vertex set is the hypergraph over vertex set $W$ with edge set consisting of all edges of $\mathcal{H}$ that are contained in $W$.
A hypergraph is a $d$-uniform hypergraph if every edge contains exactly $d$ vertices. We call a collection $\{E_1, \ldots, E_l\}$ of edges in $\mathcal{H}$ an induced matching if they form a matching in $\mathcal{H}$ (i.e., they are pairwise disjoint), and they are exactly the edges of the induced subhypergraph of $\mathcal{H}$ over the vertices contained in $\displaystyle \bigcup_{i=1}^l E_i$. The {\it induced matching number} of $\mathcal{H}$, denoted by $\im(\mathcal{H})$, is the maximum size of an induced matching in $\mathcal{H}$.

Now, we recall some notions from graph theory. Let $G$ be a finite simple graph. A subset $W$ of $V(G)$ is called an {\it independent subset} of $G$ if there are no edges among the vertices of $W$. An independent subset $W$ of $G$ is a {\it maximal independent subset} if $W \cup \{x\}$ is not an independent subset of $G$ for every vertex $x \in V(G) \setminus W$. A graph $G$ without isolated vertices is {\it very well-covered} if $n$ is an even number and every maximal independent subset of $G$ has cardinality $\frac{n}{2}$. A set of vertices $C$ of $G$ is called a {\it vertex cover} of $G$ if $C\cap e\neq \emptyset$ for all $e\in E(G)$. A {\it minimal vertex cover} is a vertex cover which is minimal with respect to set inclusion.
The {\it cover ideal} of $G$, denoted by $J(G)$, is defined as,
\begin{align*}
J(G) := \bigg( \prod_{x_i \in C} x_i \mid C \text{ is a minimal vertex cover of } G \bigg )= \bigcap_{\{x_i,x_j\} \in E(G)}(x_i,x_j).
\end{align*}

%\begin{lemma}\cite[Theorem 3.8.]{s23} \label{inequality}
%Let $G$ be a simple graph. Then $\v(J(G)) \leq \text{reg}(R/J(G))$. 

%\end{lemma}

%\begin{lemma}\cite[Corollary 3.9.]{s23}\label{cm} 
%Let $G$ be a simple graph. Then $R/I(G)$ is Cohen-Macaulay if and only if
%\begin{equation*}
  %  \v(J(G)) = \reg(R/J(G)) = \alpha_0(G) - 1. 
%\end{equation*}
%\end{lemma}

%\begin{lemma}\cite[Theorem 3.4]{f21} \label{linear-resolution}
%$G$ is a Cohen-Macaulay very well-covered graph if and only if  $J(G)^{(k)}$ \text{ has a linear resolution for some (equivalently, for all) integer } $k \geq 2$.    
%\end{lemma}

\section{Asymptotic $\v$-number of Noetherian graded filtration}\label{limit-existence}
In this section, we prove that $\displaystyle\lim_{k \rightarrow \infty}\frac{\v(I_{[k]})}{k}$ exists for any Noetherian graded filtration of ideals over a Noetherian graded domain $R$. 
\begin{lemma} \label{there-exist}
 Let $\mathcal{I}=\{I_{[k]}\}_{k \geq 0}$ be a Noetherian graded filtration of ideals. Then, there exists a positive integer $r$ such that $I_{[rk]}=(I_{[r]})^k$ for all $k \geq 1$.   
\end{lemma}
\begin{proof}
The result follows from \cite[Theorem 2.7.]{r79} and \cite[Remark 2.4.4.]{r79}.
\end{proof}

\begin{proposition} \label{colon}
    Let $I$ be a graded ideal in a Noetherian $\mathbb{N}$-graded domain $R$. If $f \in R$ is a homogeneous element such that $f\notin I$, then $\v(I)\leq \v(I:f)+\deg(f)$.
\end{proposition}

\begin{proof}
    The proof is similar to the proof of \cite[Proposition 4.3]{fs23}, where the result had been derived when $R$ is a polynomial ring.
\end{proof}

\begin{lemma}\label{limit-exists}
Let $\mathcal{I}=\{I_{[k]}\}_{k \geq 0}$ be a Noetherian graded filtration of a Noetherian graded domain $R$. Then, $\displaystyle \lim_{k \rightarrow \infty}\frac{\alpha(I_{[k]})}{k}$ exists. In particular, 
$\displaystyle \lim_{k \rightarrow \infty}\frac{\alpha(I_{[k]})}{k} = \frac{\alpha(I_{[r]})}{r}$ for some positive integer $r$.
\end{lemma}
\begin{proof}
By Lemma \ref{there-exist}, there exists $r$ such that $I_{[rn]}=(I_{[r]})^n$. Then, we have  $\alpha(I_{[rn]})=n\alpha(I_{[r]})$.
Note that $\alpha(I_{[rn]}) \leq \alpha(I_{[rn+j]})$ as $I_{[rn]} \supseteq I_{[rn+j]}$ for all $j \geq 0$. Again, $I_{[rn]}I_{[j]}\subseteq I_{[rn+j]}$ implies $\alpha(I_{[rn+j]}) \leq \alpha(I_{[rn]})+\alpha(I_{[j]})$. Thus, for any integer $k=rn+j$ with $j\geq 0$, we have 
$$\alpha(I_{[rn]}) \leq \alpha(I_{[rn+j]}) \leq \alpha(I_{[rn]})+\alpha(I_{[j]}).$$ 
This implies that for a fix $j\geq 0$,
$$\lim_{n \rightarrow \infty }\frac{n\alpha(I_{[r]})}{rn+j} \leq \lim_{n \rightarrow \infty }\frac{\alpha(I_{[rn+j]})}{rn+j} \leq \lim_{n \rightarrow \infty} \frac{n\alpha(I_{[r]})}{rn+j}+\lim_{n \rightarrow \infty}\frac{\alpha(I_{[j]})}{rn+j}.$$ 
Therefore, for a fix $j\geq 0$, $\displaystyle \lim_{n \rightarrow \infty }\frac{\alpha(I_{[rn+j]})}{rn+j}$ exists, and 
\begin{align}\label{eqalpha}
    \displaystyle\lim_{n \rightarrow \infty }\frac{\alpha(I_{[rn+j]})}{rn+j}=\frac{\alpha(I_{[r]})}{r}.
\end{align}
Now, any $k\geq r$ can be written as $k=rn+j$ for some positive integer $n$ and some $j\in\{0,\ldots,r-1\}$. Since $\{I_{[rn]}\}_{n\geq 1}$, $\{I_{[rn+1]}\}_{n\geq 1}$, \ldots, $\{I_{[rn+(r-1)]}\}_{n\geq 1}$ forms a partition of the sequence $\{I_{[k]}\}_{k\geq r}$, from equation \eqref{eqalpha} it follows that $\displaystyle\lim_{k\rightarrow\infty} \frac{\alpha(I_{[k]})}{k}$ exists and $\displaystyle\lim_{k\rightarrow\infty} \frac{\alpha(I_{[k]})}{k}=\frac{\alpha(I_{[r]})}{r}$.
\end{proof}
\noindent
%The proof of below Lemma follows from \cite[Lemma 3.1]{bms24} and \cite[Lemma 1.1]{fs24}.

\begin{lemma}\label{stable prime}
Let $\mathcal{I}=\{I_{[k]}\}_{k \geq 0}$ be a Noetherian graded filtration of a Noetherian $\mathbb{N}$-graded domain $R$. Then, there exists $c \in \mathbb{N}$ such that $\v(I_{[k]}) \geq \alpha(I_{[k]})-c$ for all $k\geq 1$.
\end{lemma}
\begin{proof}
    Similar as in \cite{bms24}, for a graded ideal $I$ of $R$, let us define $c(I):=\max\{\alpha(\mathfrak{p})\mid \mathfrak{p}\in\mathrm{Ass}(I)\}$. Note that $\mathfrak{p}\subseteq \mathfrak{q}$ implies $\alpha(\mathfrak{q})\leq \alpha(\mathfrak{p})$. Thus, $c(I)=\max\{\alpha(\mathfrak{p})\mid \mathfrak{p}\in\mathrm{Min}(I)\}$. Following the proof of \cite[Lemma 3.1]{bms24}, we can derive that $\v(I)\geq \alpha(I)-c(I)$. Now, let $\mathfrak{p}'$ be a minimal prime of $I_{[k]}$ for some $k\geq 1$. Since $(I_{[1]})^k\subseteq I_{[k]}\subseteq \mathfrak{p}'$, $I_{[1]}\subseteq \mathfrak{p}'$, and $\mathfrak{p}'$ is a minimal prime of $I_{[1]}$ as $I_{[k]}\subseteq I_{[1]}$. Therefore, we have $\mathrm{Min}(I_{[1]})=\mathrm{Min}(I_{[k]})$ for all $k\geq 1$. Hence, $c(I_{[k]})=c$ is a constant for all $k\geq 1$, where $c=c(I_{[1]})$. This gives $\v(I_{[k]})\geq \alpha(I_{[k]})-c$ for all $k\geq 1$.
\end{proof}

\begin{theorem} \label{limit-vnumber}
Let $\mathcal{I}=\{I_{[k]}\}_{k \geq 0}$ be a Noetherian graded filtration of a Noetherian graded domain $R$. Then, $\displaystyle \lim_{k \rightarrow \infty}\frac{\v(I_{[k]})}{k}$ exists and 
$$\lim_{k \rightarrow \infty}\frac{\v(I_{[k]})}{k} =\lim_{k \rightarrow \infty}\frac{\alpha(I_{[k]})}{k}.$$
\end{theorem}

\begin{proof}
By Lemma \ref{there-exist}, there exists a positive integer $r$ such that $I_{[rn]}=(I_{[r]})^n$ for all $n\geq 1$. Now, we can write any natural number $k\geq r$ as $k=rn+j$, for some $n,j\in \mathbb{N}$ such that $0 \leq j \leq r-1$. Let $f \in I_{[r]}$ be a homogeneous element with $\deg(f)=\alpha(I_{[r]})$. Then, $f^n \in {(I_{[r]})}^n=I_{[rn]}$ and  $\deg(f^n)=n\alpha(I_{[r]})=\alpha((I_{[r]})^n)=\alpha(I_{[rn]}).$ Thus, $f^{n-1} \not \in (I_{[r]})^{n-1}I_{[r]}=(I_{[r]})^{n}=I_{[rn]}$. This implies that
$$f^{n-1} \not \in I_{[rn+j]} \subseteq I_{[rn]} \text{ for any } j\geq 0.$$
Therefore, by Proposition \ref{colon}, we have $\v(I_{[rn+j]}) \leq \v((I_{[rn+j]}:f^{n-1})) +\deg(f^{n-1})$. First, we show that $(I_{[rn+j]}:f^{n-1}) \subseteq (I_{[r(n+1)+j]}:f^{n}) $ for all $n \geq 1$. Let $g \in (I_{[rn+j]} : f^{n-1})$. Then, $gf^{n-1} \in I_{[rn+j]}$. This implies that 
$$ gf^n=gf^{n-1}f \in I_{[rn+j]}I_{[r]} \subseteq I_{[r(n+1)+j]}.$$
Thus, $g \in (I_{[r(n+1)+j]}:f_n)$. Therefore, for any fixed $j\geq 0$, we have the following increasing sequence of ideals in $R$
 \begin{align*}
 I_{[r+j]} \subseteq (I_{[2r+j]} : f) \subseteq (I_{[3r+j]} : f^2) \subseteq \cdots \subseteq (I_{[nr+j]} : f^{n-1}) \subseteq (I_{[(n+1)r+j]} : f^{n}) \subseteq \cdots.
  \end{align*}
 Since $R$ is Noetherian, there exists a positive integer $m_j$ such that 
 \begin{align*}
   (I_{[rn+j]}:f^{n-1})=(I_{[rm_j+j]}:f^{m_j-1}) \mbox{ for all } n \geq m_j.
 \end{align*}
Thus, $\v((I_{[rn+j]}:f^{n-1}))=d_j$ for all $n \geq m_j$, where $d_j=\v((I_{[rm_j+j]}:f^{m_j-1}))$ is a constant. Hence, $\v(I_{[rn+j]}) \leq d_j +(n-1)\alpha(I_{[r]})$ for all $n\geq m_j$. Again, by Lemma \ref{stable prime}, there exists positive integers $n_0$ and $c$ such that $\v(I_{[rn+j]})\geq \alpha(I_{[rn+j]})-c$ for all $n\geq n_0$ and for any $j\geq 0$. Therefore, for a fix $j\geq 0$, and for all $n\geq \max\{n_0,m_j\}$, we have
\begin{align}\label{eq1}
\alpha(I_{[rn+j]}) -c \leq \v(I_{[rn+j]}) \leq d_j+(n-1)\alpha(I_{[r]}).
\end{align}
Now, for a large enough positive integer $k$, we can write $k=rn+j$, where $n\geq 1$ and $1\leq j\leq r-1$. Corresponding to a fix $j\in\{0,\ldots,r-1\}$, let us consider the subsequence of $\{I_{[k]}\}$ of the form $\{I_{[rn+j]}\}_{n\geq\max\{n_o,m_j\}}$. Note that $rn+j\rightarrow\infty$ implies $n\rightarrow\infty$. Thus, using \eqref{eq1}, we get
    \begin{align*}
   \lim_{n\rightarrow\infty}\frac{\alpha(I_{[rn+j]}) -c}{rn+j}&\leq \lim_{n\rightarrow\infty}\frac{\v(I_{[rn+j]})}{rn+j}\leq \lim_{n\rightarrow\infty}\frac{d_j+(n-1)\alpha(I_{[r]})}{rn+j}
   \end{align*}
   \begin{align}\label{eq2}
       \text{i.e., } \lim_{n\rightarrow\infty}\frac{\alpha(I_{[rn+j]})}{rn+j}&\leq \lim_{n\rightarrow\infty}\frac{\v(I_{[rn+j]})}{rn+j}\leq \frac{\alpha(I_{[r]})}{r}.\quad
    \end{align}
Due to Lemma \ref{limit-exists}, we have $\displaystyle \lim_{n \rightarrow \infty}\frac{\alpha(I_{[rn+j]})}{rn+j}=\frac{\alpha(I_{[r]})}{r}$. Hence, it follows from \eqref{eq2} that $\displaystyle\lim_{n\rightarrow\infty}\frac{\v(I_{[rn+j]})}{rn+j}$ exists, and 
\begin{align}\label{eq3}
   \lim_{n\rightarrow\infty}\frac{\v(I_{[rn+j]})}{rn+j}=\frac{\alpha(I_{[r]})}{r}. 
\end{align}
Note that the above equality is true for any $j\in \{1,\ldots,r-1\}$. Let us consider the positive integer $k_0=\max\{n_0,m_1,\ldots,m_{r-1}\}$. Then $\{I_{[rn]}\}_{n\geq k_0}$, $\{I_{[rn+1]}\}_{n\geq k_0}$, \ldots, $\{I_{[rn+(r-1)]}\}_{n\geq k_0}$ forms a partition of $\{I_{[k]}\}_{k\geq k_0}$. Hence, by equation \eqref{eq3}, it follows that $\displaystyle\lim_{k\rightarrow\infty}\frac{\v(I_{[k]})}{k}$ exists, and $\displaystyle\lim_{k\rightarrow\infty}\frac{\v(I_{[k]})}{k}=\frac{\alpha(I_{[r]})}{r}$. 
\end{proof}

%\begin{remark}
%  From the papers \cite{fs24,as24}, we observe that the $\v$-numbers of graded Noetherian filtrations $I_{[k]}$ exhibit quasi-linear behavior. The main result of this section, Theorem \ref{limit-vnumber}, ensures that the leading coefficients of all linear polynomials are equal, and equal to $\displaystyle \lim_{k \rightarrow \infty}\frac{\alpha(I_{[k]})}{k}$.  
%\end{remark}

\begin{corollary}\label{waldscmidt-constant}
 Let $I \subseteq R$ be a graded ideal of a Noetherian graded domain $R$ such that symbolic Rees algebra $R_s(I)$ is Noetherian. Then,
 \begin{align*}
     \lim_{k \rightarrow \infty}\frac{\v(I^{(k)})}{k}=\hat{\alpha}(I),
\end{align*}
where $\hat{\alpha}(I)$ is the Waldschmidt constant of $I$.
 \end{corollary}

\begin{remark}{\rm
Let $M$ be a finitely generated graded $R$-module, $N$ a graded submodule of $M$, and let $\mathcal{I} = \{I_{[k]}\}_{k \geq 0}$ be a Noetherian graded filtration of $R$. Define $\alpha(M) := \min\{d : M_d \neq 0\}$. In \cite{fg24}, the author introduced the $\v$-number for modules in a natural way. Then under suitable assumptions and modifications, it can be shown that the $\v$-function of $\{I_{[k]}M/I_{[k+1]}N\}_{k \geq 0}$ and $\{M/I_{[k+1]}N\}_{k\geq 0}$ are quasi-linear for $k>>0$ as mentioned in \cite[Remark 1.9]{fs24}. Consequently, it becomes feasible to extend Theorem \ref{limit-vnumber} to filtrations $\{I_{[k]}M/I_{[k+1]}N\}_{k \geq 0}$ and $\{M/I_{[k+1]}N\}_{k\geq 0}$, with modifications made to the proof and under analogous appropriate conditions. That is, one can establish that 
$$\displaystyle \lim_{k \rightarrow \infty}\frac{1}{k} \v\bigg(\frac{I_{[k]}M}{I_{[k+1]}N}\bigg)=\lim_{k \rightarrow \infty}\frac{1}{k}\alpha\bigg(\frac{I_{[k]}M}{I_{[k+1]}N}\bigg) \text{ and } \displaystyle \lim_{k \rightarrow \infty}\frac{1}{k}\v\bigg(\frac{M}{I_{[k+1]}N}\bigg)=\lim_{k \rightarrow \infty}\frac{1}{k} \alpha \bigg(\frac{M}{I_{[k+1]}N}\bigg).$$

\noindent \textbf{N.B.} Recently, in \cite{as24}, the authors proved that for an equigenerated graded ideal $I$ of $R$, if the filtration $\mathcal{I}=\{I^{(k)}\}_{k\geq 0}$ is Noetherian, then $\displaystyle\lim_{k\rightarrow \infty}\frac{\v_{\mathfrak{p}}(I^{(k)})}{k}$ exists for any stable prime ideal $\mathfrak{p}$ of $I$.
\medskip

%Hence, one can generalize Theorem \ref{limit-vnumber} for the module filtration $\{I_{[k]}M/I_{[k+1]}N\}_{k \geq 0}$ and $\{M/I_{[k+1]}N\}$, with adjustments made in the proof and under similar appropriate assumptions, i.e. one can prove that  
}
% By \cite[Theorem 2.11.]{fg24}, $\v\bigg(\frac{I^{k}M}{I^{k}N}\bigg)$ and  $\v\bigg(\frac{M}{I^k{N}}\bigg)$ 
\end{remark}

\section{The $\v$-number of symbolic powers of monomial ideals}\label{square-free-monomial-ideal}
In this section, we first show that for non-equigenerated square-free monomial ideals, the $\v$-number of symbolic powers is eventually smaller than regularity. Additionally, we provide an upper bound on the $\v$-number of symbolic powers for square-free monomial ideals. As a consequence, we compare the $\v$-number and regularity of all symbolic powers of certain classes of square-free monomial ideals.

Now recall the following notions from \cite{cehh17}. For a monomial ideal $I$, the {\it Newton Polyhedron} of $I$, denoted $\text{NP}(I)$, defined as 
$\text{NP}(I) := \text{convex hull}(\{ a \in \mathbb{N}^n \mid x^a \in I \}).$ Let $I$ has a minimal primary decomposition
$I = Q_1 \cap \dots \cap Q_s \cap Q_{s+1} \cap \dots \cap Q_t$, 
where $Q_1, \dots, Q_s$ are all the primary monomial ideals associated with the minimal prime ideals of $I$. Then the {\it Symbolic Polyhedron} associated to $I$, denoted $\mathcal{SP}(I)$, is defined as,
\begin{align*}
     \mathcal{SP}(I) := \text{NP}(Q_1) \cap \dots \cap \text{NP}(Q_s) \subset \mathbb{R}^r, 
\end{align*}
where $\text{NP}(Q_i)$ is the Newton Polyhedron of $Q_i$. For a vector $v = (v_1, \dots, v_r) \in \mathbb{R}^r$, denote $|v| = v_1 + \dots + v_r$. Let $\delta(I) = \max\{ |v| : v \text{ is a vertex of } \mathcal{SP}(I) \}.$ Also, $\hat{\alpha}(I) = \min\{ |v| : v \text{ is a vertex of } \mathcal{SP}(I) \}.$ Note that $\hat{\alpha}(I) \leq \delta(I)$. 

\begin{proposition}\label{alpha-delta}
 Let $I$ be a monomial ideal of $R=\mathbb{K}[x_1,\ldots,x_n]$ such that $\hat{\alpha}(I) < \delta(I)$. Then, $\v(I^{(k)}) \leq \reg(R/I^{(k)})$ for all $k \gg 0$.   
\end{proposition}
\begin{proof}
By \cite[Theorem 3.6]{dhnt21}, we have $\displaystyle \lim_{k \rightarrow \infty}\frac{\reg(I^{(k)})}{k}=\delta(I)$  and by Corollary \ref{waldscmidt-constant}, we have $\displaystyle \lim_{k \rightarrow \infty}\frac{\v(I^{(k)})}{k} =\hat{\alpha}(I).$ Again, from \cite[Theorem 1.3.]{fs24} and \cite[Corollary 3.3.]{hht07}, it follows that $\v(I^{(k)})$ and $\reg(I^{(k)})$ are quasi-linear. Therefore, there exist non-negative integers $k_1$, $d_1$ and integers $b_1,\ldots,b_{d_1}$  such that
$\v(I^{(k)})=\hat{\alpha}(I)k+b_i, \text{for all } k\geq k_1, \text{ where } i \equiv k \pmod{d_1}.$
Similarly, there exist non-negative integers $k_2$, $d_2$ and integers $c_1,\ldots,c_{d_2}$ such that $\reg(R/I^{(k)})=\delta(I)k+c_j, \text{ for all } k\geq k_2, \text{ where }  j \equiv k \pmod{d_2}$. Hence, by choosing $k\geq \max\Big\{\frac{b_i-c_j}{\delta(I)-\hat{\alpha}(I)},k_1,k_2: 1 \leq i \leq d_1, 1\leq j \leq d_2\Big\},$ we get 
$$\reg(R/I^{(k)})-\v(I^{(k)}) =((\delta(I)-\hat{\alpha}(I))k+c_j-b_i) \geq 0,$$ where $i \equiv k \pmod{d_1}$ and $j \equiv k \pmod{d_2},$ as required. 
%Thus, $$\v(I^{(k)}) \leq \reg(R/I^{(k)}), \text{ for all } k\geq \max\Big\{\frac{b_i-c_j}{\delta(I)-\hat{\alpha}(I)},k_1,k_2: 1 \leq i \leq d_1, 1\leq j \leq d_2\Big\}.$$   
\end{proof}

\begin{corollary} \label{nonequi}
Let $I$ be a non-equigenerated square-free monomial ideal in $R$. Then, $\v(I^{(k)}) \leq \reg(R/I^{(k)})$ for all $k \gg 0$.     
\end{corollary}
\begin{proof}
By \cite[Proposition 4.8.]{cdffhsty22}, $\hat{\alpha}(I) \leq \alpha(I)$ and by \cite[Lemma 4.3]{dhnt21}, $w(I) \leq \delta(I)$, where $w(I)$ is the maximal degree of minimal generators of $I$. Since $I$ is non-equigenerated, $\alpha(I) < w(I)$, and this implies $\hat{\alpha}(I)<\delta(I)$. Therefore, by Proposition \ref{alpha-delta}, it follows that $\v(I^{(k)}) \leq \reg(R/I^{(k)})$ for all $k \gg 0$.   
\end{proof}

In \cite{nkrt22}, the authors define the symbolic strong persistence property by taking the definition of symbolic power as in Definition \ref{def-symbolic} (2). However, it can be defined by considering any definition of symbolic power as follows:

\begin{definition}{\rm
    Let $I$ be an ideal in a Noetherian ring $R$. Then $I$ is said to have the \textit{symbolic strong persistence property} if 
    $(I^{(k)}:I^{(1)})=I^{(k-1)} \text{ for all } k\geq 1.$
    }
\end{definition}

\begin{theorem}\label{thm_vsym_ub}
    Let $I$ be a graded ideal in a $\mathbb{N}$-graded Noetherian domain $R$. If $I$ has the symbolic strong persistence property, then
    $\v(I^{(k)})\leq (k-1)d(I^{(1)})+\v(I^{(1)}), \text{ for all } k \geq 1$. 
\end{theorem}
\begin{proof}
    Let $\{f_1,\ldots,f_r\}$ be a minimal generating set of $I^{(1)}$. Let $g\in R$ be a homogeneous element such that $(I^{(k-1)}:g)=\mathfrak{p}$ for some $\mathfrak{p}\in\mathrm{Ass}(I^{(k-1)})$ and $\v(I^{(k-1)})=\deg(g)$. Now, due to the symbolic strong persistence property of $I$, we have
    %\begin{align*}
       % & (I^{(k-1)}:g)=\mathfrak{p}\\
       % \implies & ((I^{(k)}:I^{(1)}):g)=\mathfrak{p}\\
       % \implies & %\bigg(\bigcap_{i=1}^{r}\big(I^{(k)}:f_i\big):g\bigg)=\mathfrak{p}\%\
      %  \implies & %\bigcap_{i=1}^{r}\big((I^{(k)}:f_i):g\big)=\mathfrak{p}\\
      %  \implies & \bigcap_{i=1}^{r}\big(I^{(k)}:f_ig\big)=\mathfrak{p}.
    %\end{align*}
 \begin{align*}
 \mathfrak{p}=(I^{(k-1)}:g)=((I^{(k)}:I^{(1)}):g)=\bigg(\bigcap_{i=1}^{r}\big(I^{(k)}:f_i\big):g\bigg)= \bigcap_{i=1}^{r}\big((I^{(k)}:f_i):g\big)= \bigcap_{i=1}^{r}\big(I^{(k)}:f_ig\big).
 \end{align*}   
    This gives $(I^{(k)}:f_{i}g)=\mathfrak{p}$ for some $1\leq i\leq r$. Therefore, $\v(I^{(k)})\leq \deg(f_i)+\deg(g)\leq d(I^{(1)})+\v(I^{(k-1)})$. Now, repeated application gives $\v(I^{(k)})\leq (k-1)d(I^{(1)})+\v(I^{(1)})$.
\end{proof}

\begin{corollary}
    Let $I$ be a graded ideal with no embedded prime in a $\mathbb{N}$-graded Noetherian domain $R$. If $I$ has the symbolic strong persistence property, then
    $\v(I^{(k)})\leq (k-1)d(I)+\v(I)$, for all $k\geq 1$.
\end{corollary}
\begin{proof}
    Note that if $I$ has no embedded prime, then both the definition of symbolic powers given in Definition \ref{def-symbolic} are the same, and so, $I^{(1)}=I$. Hence, the result follows from Theorem \ref{thm_vsym_ub}.
\end{proof}

\begin{corollary}\label{corsqfreeub}
     Let $I$ be a square-free monomial ideal in $R=\mathbb{K}[x_1,\ldots,x_n]$. Then
    $\v(I^{(k)})\leq (k-1)d(I)+\v(I)$, for all $k\geq 1$.
\end{corollary}
\begin{proof}
For any square-free monomial ideal $I$, \cite[Theorem 5.1]{nkrt22} establishes that $(I^{(k)} : I) = I^{(k-1)}$ holds for all $k \geq 1$. Thus, the result follows from Theorem \ref{thm_vsym_ub}.
\end{proof}

\begin{remark}{\rm
Let $I$ be a graded ideal in a $\mathbb{N}$-graded Noetherian domain $R$. If $I$ has the strong persistence property, then $\v(I^k)\leq (k-1)d(I)+\v(I)$ for all $k\geq 1$. The proof is similar to the proof of Theorem \ref{thm_vsym_ub}.
}
\end{remark}

\begin{remark}\label{remvsymb} {\rm
    Let $I\subseteq R$ be a square-free monomial ideal. Then by \cite[Proposition 2.2]{f23}, we know $\v(I)\geq \alpha(I)-1$. There are many instances where equality holds (see \cite{bms24}, \cite{f23}). Since $d(I)\geq \alpha(I)$, there are several square-free monomial ideal $I$ with $\v(I)\leq d(I)-1$. Now, by \cite[Lemma 4.2(ii)]{dhnt21}, we have $d(I)k\leq d(I^{(k)})$ for all $k\geq 1$. Hence, $\reg(R/I^{(k)})\geq kd(I)-1$ for all $k\geq 1$. Thus, as a consequence of Corollary \ref{corsqfreeub}, it follows that if $\v(I)\leq d(I)-1$, then $\v(I^{(k)})\leq kd(I)-1\leq \reg(R/I^{(k)}) \text{ for all } k\geq 1$.
    }
\end{remark}

\begin{definition}{\rm
    A polymatroidal ideal $I$ is a monomial ideal, generated in a single degree, satisfying the following ``exchange conditions": For two minimal monomial generators $u = x_1^{a_1} \cdots x_n^{a_n}$ and $v = x_1^{b_1} \cdots x_n^{b_n}$ of $I$, and for each $i$ with $a_i > b_i$, there exists $j$ such that $a_j < b_j$ and $\frac{x_j u}{x_i}$ is a minimal generator of $I$.
    }
\end{definition}

\begin{proposition}\label{prop_poly}
Let $I$ be a square-free polymatroidal ideal generated in degree $d$. Then 
\begin{align*}
    \v(I^{(k)}) \leq \v(I^k) =dk-1= \reg(R/I^k)\leq \reg(R/I^{(k)}) \text{ for all } k\geq 1.
\end{align*}
Moreover, if $I^{(k)}$ has a linear resolution for some $k\geq 2$, then 
\begin{align*}
    \v(I^{(k)}) = \v(I^k) =dk-1= \reg(R/I^k)= \reg(R/I^{(k)}) \text{ for all } k\geq 1.
\end{align*}
\end{proposition}
\begin{proof}
By \cite[Theorem 5.5]{f23}, $\v(I^k)=dk-1$ for all $k\geq 1$. Thus, by Corollary \ref{corsqfreeub}, we have $\v(I^{(k)})\leq dk-1$ for all $k\geq 1$. Again, as stated in Remark \ref{remvsymb}, $\reg(R/I^{(k)})\geq dk-1$ for all $k\geq 1$. It is well-known that polymatroidal ideals have linear powers. Hence, the first part of the result follows.\par 
Now, suppose $I^{(k)}$ has a linear resolution for some $k\geq 2$. Then, by \cite[Theorem 4.7]{mt19}, $I^{(k)}$ has a linear resolution for all $k\geq 1$. Thus, $\reg(R/I^{(k)})=dk-1$ and due to \cite[Lemma 4.2(ii)]{dhnt21} $\alpha(I^{(k)})=dk$. Since associated primes of a monomial ideal are generated by variables, then by the notation of the Lemma \ref{stable prime}, we get $c=1$ and $\mathrm{v}(I^{(k)})\geq dk-1$ for all $k\geq 1$. Hence, the second part follows.
\end{proof}

Next, we give more general conditions for which $\v(I^{(k)})\leq \reg(R/I^{(k)})$ and $\v(I^{k})\leq \reg(R/I^{k})$ hold for all $k\geq 1$.

\begin{proposition}\label{propvimreg}
    Let $I=I(\mathcal{H})$ be a square-free monomial ideal in $R=\mathbb{K}[x_1,\ldots,x_n]$, where $\mathcal{H}$ is the corresponding hypergraph of $I$. Let $E_1, \ldots, E_s$ be an induced matching of $\mathcal{H}$ with $\vert E_1\vert=d(I)$. If $\v(I) \leq \sum_{i=1}^{s}(\vert E_i\vert -1)$, then $\v(I^{(k)}) \leq \reg(R/I^{(k)})$ for all $k \geq 1$. Moreover, if $I$ has the strong persistence property, then $\v(I^{k}) \leq \reg(R/I^{k})$ for all $k \geq 1$ provided $\v(I) \leq \sum_{i=1}^{s}(\vert E_i\vert -1)$.
    \end{proposition}

\begin{proof}
    By \cite[Theorem 3.7]{bcdms22}, we have $kd(I)-1+\sum_{i=2}^{s}(\vert E_i\vert-1)\leq \reg(R/I^{(k)})$. Now, by the given condition, $$kd(I)-1+\sum_{i=2}^{s}(\vert E_i\vert-1)=(k-1)d(I)+\sum_{i=1}^{s}(\vert E_i\vert -1)\geq (k-1)d(I)+\v(I).$$
    Hence, the assertion follows from Corollary \ref{corsqfreeub}.
\end{proof}

\begin{corollary}\label{uniform-hypergraph-matching}
Let $\mathcal{H}$ be a $d$-uniform hypergraph and $I=I(\mathcal{H})$ be such that $\v(I) \leq \im(\mathcal{H})(d-1)$. Then $\v(I^{(k)}) \leq \reg(R/I^{(k)})$ for all $k \geq 1$.
\end{corollary}
\begin{proof}
Since $\mathcal{H}$ is $d$-uniform, $d(I)=d$, and the result holds due to Proposition \ref{propvimreg}.
\end{proof}

\begin{corollary}\label{corvregedge}
    Let $G$ be a simple graph. Then $\v(I(G)^{k})\leq 2(k-1)+\v(I(G))$ and $\v(I(G)^{(k)})\leq 2(k-1)+\v(I(G))$ for all $k\geq 1$. If $\v(I(G))\leq \im(G)$, then $\v(I(G)^k)\leq \reg(R/I(G)^k)$ and $\v(I(G)^{(k)})\leq \reg(R/I(G)^{(k)})$ for all $k\geq 1$.
\end{corollary}
\begin{proof}
    The proof follows from Proposition \ref{propvimreg} and the fact that $I(G)$ has the strong persistence property (see \cite[Lemma 2.12]{mmv12}).
\end{proof}

\begin{remark}\label{bipartite-chordal}{\rm
 By \cite[Theorem 4.5, Theorem 4.10, Theorem 4.11, Theorem 4.12]{ss22}, if $G$ is bipartite graphs, or chordal graphs, or $(C_4, C_5)$-free vertex decomposable graphs, or whiskers graphs, then we have $\v(I(G)) \leq \im(G)$. Also, by \cite[Corollary 4, Theorem 12, Theorem 13, Theorem 14]{grv21}, if $G$ is very well-covered, or $G$ has a simplicial partition, or $G$ is well-covered connected and contains neither four nor five cycles, or $G$ is a cycle $C_n$ with $n \neq 5$, then $\v(I(G)) \leq \im(G)$.  Thus, by Corollary \ref{corvregedge}, for these classes of graphs, $\v(I(G)^{(k)}) \leq  \reg(R/I(G)^{(k)})$ and $\v(I(G)^{k}) \leq  \reg(R/I(G)^{k})$ for all $k \geq 1$.
     }
\end{remark}

\section{The $\v$-number and regularity of symbolic powers of cover ideals} \label{cover-ideal}
This section compares the $\v$-number and regularity of symbolic powers of cover ideals of simple graphs. Note that using the previous section's result, it is not possible to compare $\v$-number and regularity of cover ideals of all graphs. Thus, we use \cite[Theorem 3.8.]{s23} and Seyed Fakhari's construction \cite{f18} to prove our main results. Let us start with Seyed Fakhari's construction. 
%In \cite{f18}, Fakhari constructed a new graph $G_k$ whose vertex cover ideal is strongly related to the $k$-th symbolic power of the vertex cover ideal of $G$.
\medskip

\noindent \textbf{Construction.} Let $G$ be a graph with vertex set $V(G) = \{x_1, \ldots, x_n\}$. For an integer $k \geq 1$, let us define a new graph $G_k$ as follows:
\begin{enumerate}
    \item[$\bullet$] $V(G_k) = \{x_{i,p} \mid 1 \leq i \leq n \text{ and } 1 \leq p \leq k\}$,

    \item[$\bullet$] $E(G_k) = \{\{x_{i,p}, x_{j,q}\} \mid \{x_i, x_j\} \in E(G) \text{ and } p + q \leq k + 1\}$.
\end{enumerate}
\medskip

We start by stating the following two lemmas that will be used to prove the main results of this section.

\begin{lemma}\cite[Lemma 3.4]{f18} \label{polarisation}
  Let $G$ be a graph. For every integer $k \geq 1$, the ideal $(J(G)^{(k)})^{\PP}$ is the cover ideal of $G_k$.  
\end{lemma}

We say a graph $G$ is \textit{Cohen-Macaulay} if $R/I(G)$ is Cohen-Macaulay. $G$ is called \textit{Cohen-Macaulay very well-covered} if $G$ is both Cohen-Macaulay and very well-covered.

\begin{lemma}\cite[Proposition 3.1]{f18} \label{very-well}
Let $G$ be a graph without isolated vertices, and $k \geq 1$ be an integer. Then, the following holds.
\begin{enumerate}
    \item If $G$ is very well-covered, then $G_k$ is very well-covered too.
    \item If $G$ is Cohen-Macaulay and very well-covered, then $G_k$ is Cohen-Macaulay too.
\end{enumerate}
\end{lemma}

\begin{theorem}\label{v-numver-less-regularity}
 Let $G$ be a simple graph. Then,
 \begin{align*}
    \v(J(G)^{(k)}) \leq \reg(R/J(G)^{(k)}) \text{ for all } k \geq 1. 
 \end{align*}
\end{theorem}
\begin{proof}
Note that $(J(G)^{(k)})^{\PP}=J(G_k)$ by Lemma \ref{polarisation}. Also, thanks to \cite[Theorem 3.8.]{s23}, we have  $\v(J(G_k)) \leq \reg(J(G_k))-1$. Again, due to Lemma \ref{embedded-primes}, we have
$$\v(J(G)^{(k)})=\v((J(G)^{(k)})^{\PP})=\v(J(G_k)).$$ 
Therefore, using all the above together and \cite[Corollary 1.6.3.]{hh11}, we get
$$\v(J(G)^{(k)})=\v(J(G_k)) \leq \reg(J(G_k))-1=\reg((J(G)^{(k)})^{\PP})-1=\reg(J(G)^{(k)})-1.$$
This implies that $\v(J(G)^{(k)}) \leq \reg(R/J(G)^{(k)})$ for all $k\geq 1$.
\end{proof}

\begin{theorem}\label{v-number-equal-regularity}
 Let $G$ be a simple graph. Then $G$ is a Cohen-Macaulay very-well covered graph if and only if  
 \begin{align*}
    \v(J(G)^{(k)})=\reg(R/J(G)^{(k)})= \alpha(J(G)^{(k)})-1 \text{ for all } k \geq 1.
 \end{align*}
\end{theorem}

\begin{proof}
Note that by Lemma \ref{polarisation}, we have $(J(G)^{(k)})^{\PP}=J(G_k)$. First, assume $G$ is a Cohen-Macaulay very well-covered graph. Then, by Lemma \ref{very-well}, $G_k$ is Cohen-Macaulay. Thus, it follows from \cite[Corollary 3.9.]{s23} that  $$\v(J(G_k))=\reg(J(G_k))-1=\alpha_{0}(G_k)-1.$$
Again, due to Lemma \ref{embedded-primes}, $\v(J(G)^{(k)})=\v(J(G)^{(k)})^{\PP})=\v(J(G_k)).$ Thus, by \cite[Corollary 1.6.3.]{hh11},
$$\v(J(G)^{(k)})=\v(J(G_k))=\reg(J(G_k))-1=\reg((J(G)^{(k)})^{\PP})-1=\reg(J(G)^{(k)})-1.$$
Therefore, we have $ \v(J(G)^{(k)})=\reg(R/J(G)^{(k)})= \alpha(J(G)^{(k)})-1$.\par 

Conversely, we assume that $\v(J(G)^{(k)})=\reg(R/J(G)^{(k)})=\alpha(J(G)^{(k)})-1$. Then, by Lemma \ref{polarisation}, $\v(J(G_k))=\reg(R/J(G_k))=\alpha_{0}(G_k)-1$. This implies $R/I(G_k)$ is Cohen-Macaulay by \cite[Corollary 3.9]{s23}. Thus, $J(G_k)=(J(G)^{(k)})^{\PP}$ has a linear resolution (see \cite[Theorem 8.1.9]{hh11}), which gives $J(G)^{(k)}$ has linear resolution by \cite[Corollary 1.6.3]{hh11}. This holds for any $k\geq 1$, and hence, $G$ is a Cohen-Macaulay very well-covered graph due to \cite[Theorem 3.4]{f21}.
\end{proof}

% \begin{lemma}\label{power-element}
% Let $I \subseteq R$ be a square-free monomial ideal. Then for any $f \in \mathcal{G}(I)$, $f^n \not \in \mathcal{G}(I^{(n+1)})$.   
% \end{lemma}
% \begin{proof}
 % Suppose $f=x_1x_2 \ldots x_r \in \mathcal{G}(I) $. If possible, suppose $f^n=x_1^nx_2^n \ldots x_r^n \in I^{(n+1)}$. By Lemma \ref{differential-power}, $I^{(n)}=I^{[n]}$. This implies that $\frac{\partial^{\ast}f^n}{\partial^{\ast}x^a} \in I$ for all $\mid \mathfrak{a}\mid \leq n$. Take $\mathfrak{a}=(n,0,\ldots,0)$. Then $x_2^n \ldots x_r^n \in I $. This implies that $x_2\ldots x_r \in I$ as $I$ is a radical ideal. Thus, we get a contradiction because $x_1x_2\ldots x_r \in \mathcal{G}(I)$.    
%\end{proof}

%\begin{theorem}\label{upper-bound-cover-ideal}
%Let $G$ be a simple graph. Then, 
%\begin{align*}
   %\v(J(G)^{(2k)}) &\leq \v(J(G))+\alpha(G)+(k-1) \mid V(G) \mid \\
  % \v(J(G)^{(2k+1)}) &\leq \v(J(G))+k \mid V(G) \mid.
%\end{align*}   
%\end{theorem}
%\begin{proof}
%By using Lemma \ref{symbolic-cover} and Proposition \ref{colon} repeatedly, we have
%\begin{align*}
 %  \v(J(G)^{(2k)}) &\leq \v(J(G)^{(2)})+(k-1) \mid V(G) \mid \\
 %  \v(J(G)^{(2k+1)}) &\leq \v(J(G))+k \mid V(G) \mid.
%\end{align*}
%\textcolor{red}{ Take $f \in \mathcal{G}(J(G))$. Then, by Lemma \ref{power-element}, $f \not \in J(G)^{(2)}$. Note that $J(G) \subseteq (J(G)^{(2)}:f)$. Also, $\v((J(G)^{(2)}:f)) \leq \v(J(G))$. Therefore by Proposition \ref{colon}, 
%$$ \v(J(G)^{(2)}) \leq \v(J(G))+\alpha(G).$$ Thus, we get the required result. }  
%\end{proof}

\begin{proposition}\label{second-alpha}
 Let $G$ be a simple graph. Then, 
 \begin{align*}
 \lim_{k \rightarrow \infty} \frac{ \v(J(G)^{(k)}) }{k}=\frac{\alpha(J(G)^{(2)})}{2}.   
 \end{align*}
\end{proposition}
\begin{proof}
The proof follows from Theorem \ref{limit-vnumber} and \cite[Corollary 4.4]{dg20}. 
\end{proof}

\begin{corollary}\label{minimum-vertex-cover}
Let $G$ be a graph such that for every odd cycle $C$ of $G$ and for every $i\in V(G)$, there exists a vertex $j$ of $C$ with $\{i, j\} \in E(G)$. Then, 
\begin{align*}
 \lim_{k \rightarrow \infty} \frac{ \v(J(G)^{(k)}) }{k}= \min \bigg\{\alpha(J(G)),\frac{ \mid V(G) \mid}{2} \bigg\}.  
 \end{align*} 
\end{corollary}
\begin{proof}
 By \cite[Proposition 5.3]{hht07}, we have $J(G)^{(2)}=J(G)^2+(x_1 \cdots x_n)$. This implies that 
 $$\frac{\alpha(J(G)^{(2)})}{2}= \min \bigg\{\alpha(J(G)),\frac{\mid V(G) \mid}{2} \bigg\}.$$
 Thus, by Proposition \ref{second-alpha}, we get the required result. 
\end{proof}

\begin{remark}{\rm
The complete graphs $K_n$, cyclic graphs $C_n$, and a graph $G$ obtained from $C_n$ by adding leaves to any vertices of $C_n$, all belong to the class of graphs described in Corollary \ref{minimum-vertex-cover}. In fact, we have $\displaystyle \lim_{k \rightarrow \infty} \frac{ \v(J(G)^{(k)}) }{k}=\frac{ \lvert V(G) \rvert}{2}$, where $G$ is a complete graph or a cyclic graph $C_n$, or a graph obtained from $C_n$ by adding leaves to any vertices, provided that the total number of leaves in $G$ is at most $\frac{ \lvert V(G) \rvert}{2}$.
}
\end{remark}

\noindent 
{\bf Acknowledgement:} 
Manohar Kumar is thankful to the Government of India for supporting him in this work through the Prime Minister Research Fellowship. Kamalesh Saha would like to thank the National Board for Higher Mathematics (India) for the financial support through the NBHM Postdoctoral Fellowship. Again, Kamalesh Saha is partially supported by an Infosys Foundation fellowship.


\begin{thebibliography}{AAAA}

\bibitem{ass23} S. B. Ambhore, K. Saha and I. Sengupta, {\em The $\mathrm{v}$-Number of Binomial Edge Ideals}, to appear in {Acta Math. Vietnam.}, arXiv:2304.06416 (2023). 

\bibitem{as24} Vamathi A and P. Sarkar, {\em $\v$-numbers of symbolic power filtrations}, arXiv preprint \url{https://doi.org/10.48550/arXiv.2403.09175} (2024).

\bibitem{bcdms22} A. Banerjee, B. Chakraborty, K. Das, M. Mandal and  S.  Selvaraja, {\em Regularity of powers of square-free monomial ideals}, {J. Pure Appl. Algebra}. \textbf{226}, Paper No. 106807, 12 (2022).

\bibitem{bcghjsvv16} C. Bocci, S. Cooper, E. Guardo, B. Harbourne, M.  Janssen, U. Nagel, A. Seceleanu, A. Van Tuyl, and T. Vu, {\em The Waldschmidt constant for square-free monomial ideals}, {J. Algebraic Combin.}. \textbf{44}, 875-904 (2016).

\bibitem{bm23} P. Biswas and M. Mandal {\em A study of $\v $-number for some monomial ideals}, arXiv preprint \url{https://doi.org/10.48550/arXiv.2308.08604} (2023).

\bibitem{bh10} C. Bocci and B. Harbourne, {\em Comparing powers and symbolic powers of ideals}, {J. Algebraic Geom.} \textbf{19}, 399-417 (2010).


\bibitem{bms24} P. Biswas, M. Mandal, and K. Saha, {\em Asymptotic behaviour and stability index of v-numbers of graded ideals,} arXiv preprint \url{https://doi.org/10.48550/arXiv.2402.16583} (2024).

\bibitem{c23}  Y. Civan, {\em The v-number and Castelnuovo-Mumford regularity of graphs}, {J. Algebraic Combin.}. \textbf{57}, 161-169 (2023).

\bibitem{cdffhsty22} J. Camarneiro, B. Drabkin, D. Fragoso, W. Frendreiss, D. Hoffman, A. Seceleanu, T. Tang, T and S. Yang, {\em Convex bodies and asymptotic invariants for powers of monomial ideals}, {J. Pure Appl. Algebra}. \textbf{226}, Paper No. 107089, 21 (2022).

\bibitem{cehh17}  S. Cooper, R. Embree, H. Hà and  A.  Hoefel, {\em  Symbolic powers of monomial ideals}, {Proc. Edinb. Math. Soc. (2)}. \textbf{60}, 39-55 (2017).

\bibitem{co23} A. Conca, {\em A note on the $v$-invariant}, to appear in {Proc. Amer. Math. Soc.} (2024). 

\bibitem{cstpv20} S. Cooper, A. Seceleanu, Ş Tohăneanu, M. Pinto and R. Villarreal,  {\em Generalized minimum distance functions and algebraic invariants of Geramita ideals}, {Adv. In Appl. Math.}. \textbf{112} pp. 101940, 34 (2020).

\bibitem{ddghn18}H. Dao, A. De Stefani, E. Grifo, C. Huneke and L. Núñez-Betancourt, {\em Symbolic powers of ideals}, Singularities and foliations. geometry, topology and applications, 387–432. Springer Proc. Math. Stat., 222. Springer, Cham (2018). 

\bibitem{dg20} B. Drabkin and L. Guerrieri, {\em Asymptotic invariants of ideals with Noetherian symbolic Rees algebra and applications to cover ideals,} {J. Pure Appl. Algebra}. \textbf{224}, 300-319 (2020).

\bibitem{dhnt21} L. Dung, T. Hien, H. Nguyen and  T. Trung, {\em Regularity and Koszul property of symbolic powers of monomial ideals}, {Math. Z.} \textbf{298}, 1487-1522 (2021).

%\bibitem{f17} S. Seyed Fakhari, {\em Depth and Stanley depth of symbolic powers of cover ideals of graphs}, {\em J. Algebra}. \textbf{492} pp. 402-413 (2017).

%\bibitem{f20} S. Seyed Fakhari, {\em Stability of depth and Stanley depth of symbolic powers of square-free monomial ideals}, {\em Proc. Amer. Math. Soc.}. \textbf{148}, 1849-1862 (2020).

\bibitem{f23} A. Ficarra, {\em Simon Conjecture and the $\v $-number of monomial ideals}, {Collect. Math.} \url{https://doi.org/10.1007/s13348-024-00441-z} (2024).

\bibitem{fg24} L. Fiorindo and D. Ghosh, {\em On the asymptotic behaviour of the Vasconcelos invariant for graded modules}, arxiv preprint \url{https://doi.org/10.48550/arXiv.2401.16358} (2024).

\bibitem{fs23} A. Ficarra and E. Sgroi, {\em Asymptotic behaviour of the v-number of homogeneous ideals}, arXiv preprint \url{https://doi.org/10.48550/arXiv.2306.14243} (2023).

\bibitem{fs24} A. Ficarra and E. Sgroi, {\em Asymptotic behaviour of integer programming and the $\v $-function of a graded filtration}, arXiv preprint \url{https://doi.org/10.48550/arXiv.2403.08435} (2024).

\bibitem{gkr93} A. V. Geramita, M. Kreuzer and L. Robbiano, {\em Cayley-Bacharach schemes and their canonical modules}, {Trans. Amer. Math. Soc.}. \textbf{339}, 163-189 (1993).

\bibitem{grv21} G. Grisalde, E. Reyes and R. H. Villarreal, {\em Induced matchings and the v-number of graded ideals}, {Mathematics} {\bf 9}, no. 22, 2860 (2021).

%\bibitem{gs21} E. Grifo and A. Seceleanu { \em Symbolic Rees algebras}, {\em Commutative Algebra}. pp. 343-371, (2021).

\bibitem{hh11} J. Herzog\ and\ T. Hibi, {\em Monomial ideals}, Graduate Texts in Mathematics, 260, Springer-Verlag London, Ltd., London, (2011).

\bibitem{hht07}  J. Herzog, T. Hibi and N. Trung, {\em Symbolic powers of monomial ideals and vertex cover algebras}, {Adv. Math.} \textbf{210}, 304-322, (2007).

\bibitem{jv21} D. Jaramillo and R. Villarreal, {\em The v-number of edge ideals}, {J. Combin. Theory Ser. A}. \textbf{177} pp. Paper No. 105310, 35 (2021).

\bibitem{ks23} N. Kotal and K. Saha, {\em On the $\v$-number of Gorenstein ideals and Frobenius powers}, arXiv preprint \url{https://doi.org/10.48550/arXiv.2311.04136} (2023).

\bibitem{mmv12} J. Martínez-Bernal, S. Morey and R. Villarreal, {\em Associated primes of powers of edge ideals}, {Collect. Math.} \textbf{63}, 361-374 (2012).

\bibitem{mt19} N. Minh, and T. Trung, {\em Regularity of symbolic powers and arboricity of matroids}, {Forum Math.}. \textbf{31}, 465-477 (2019).

\bibitem{n22} R. Nanduri, {\em On regularity of symbolic Rees algebras and symbolic powers of vertex cover ideals of graphs}, {Proc. Amer. Math. Soc.} \textbf{150}, no.5, 1955–1965 (2022).

\bibitem{nkrt22}  M. Nasernejad, K. Khashyarmanesh, L. Roberts and J. Toledo, {\em The strong persistence property and symbolic strong persistence property}, {Czechoslovak Math. J.} \textbf{72 (147)}, 209-237  (2022).

\bibitem{r79} L. Ratliff, {\em Notes on essentially powers filtrations}, {Michigan Math. J.} \textbf{26}, 313-324 (1979).

\bibitem{ss22} K. Saha and I. Sengupta, {\em The {$\mathrm{v}$}-number of monomial ideals}, {J. Algebraic Combin.} \textbf{56}, 903-927 (2022).

\bibitem{s23} K. Saha, {\em The v-Number and Castelnuovo-Mumford Regularity of Cover Ideals of Graphs}, {Int. Math. Res. Not. IMRN}, rnad277 (2023).

\bibitem{f18}  S. Seyed Fakhari, {\em Symbolic powers of cover ideal of very well-covered and bipartite graphs}, {Proc. Amer. Math. Soc.} \textbf{146}, 97-110 (2018).

\bibitem{f21} S. Seyed Fakhari, {\em On the minimal free resolution of symbolic powers of cover ideals of graphs}, {Proc. Amer. Math. Soc.} \textbf{149}, 3687-3698 (2021).

\bibitem{w77} M. Waldschmidt, {\em Propriétés arithmétiques de fonctions de plusieurs variables. II}, Lecture Notes in Math., vol. 578, pp. 108--135 (1977).
\end{thebibliography}
\end{document}